\documentclass[11pt]{amsart}

\usepackage[T1]{fontenc}
\usepackage{lmodern}
\usepackage{microtype}
\usepackage{amsmath,amssymb,mathtools}
\usepackage{aliascnt}
\usepackage{hyperref}
\usepackage[nameinlink,noabbrev]{cleveref}

\hypersetup{
  colorlinks=true,
  linkcolor=blue,
  citecolor=blue,
  urlcolor=blue,
  pdftitle={An Explicit Characteristic-2 Counterexample to the Separable Jacobian Conjecture}
}

\newtheorem{theorem}{Theorem}[section]

\newaliascnt{proposition}{theorem}
\newtheorem{proposition}[proposition]{Proposition}
\aliascntresetthe{proposition}

\newaliascnt{lemma}{theorem}
\newtheorem{lemma}[lemma]{Lemma}
\aliascntresetthe{lemma}

\newaliascnt{corollary}{theorem}
\newtheorem{corollary}[corollary]{Corollary}
\aliascntresetthe{corollary}

\newaliascnt{conjecture}{theorem}
\newtheorem{conjecture}[conjecture]{Conjecture}
\aliascntresetthe{conjecture}

\theoremstyle{remark}
\newaliascnt{remark}{theorem}
\newtheorem{remark}[remark]{Remark}
\aliascntresetthe{remark}

\crefname{theorem}{theorem}{theorems}
\Crefname{theorem}{Theorem}{Theorems}
\crefname{proposition}{proposition}{propositions}
\Crefname{proposition}{Proposition}{Propositions}
\crefname{lemma}{lemma}{lemmas}
\Crefname{lemma}{Lemma}{Lemmas}
\crefname{corollary}{corollary}{corollaries}
\Crefname{corollary}{Corollary}{Corollaries}
\crefname{conjecture}{conjecture}{conjectures}
\Crefname{conjecture}{Conjecture}{Conjectures}
\crefname{remark}{remark}{remarks}
\Crefname{remark}{Remark}{Remarks}

\newcommand{\A}{\mathbb{A}}
\newcommand{\Ftwo}{\mathbb{F}_{2}}

\newcommand{\Jac}{\operatorname{Jac}}
\newcommand{\Res}{\operatorname{Res}}
\newcommand{\trdeg}{\operatorname{trdeg}}

\title{An Explicit Characteristic-$2$ Counterexample to the Separable Jacobian Conjecture}

\author{Irit Huq-Kuruvilla}
\date{\today}


\begin{document}

\begin{abstract}
Let $k$ be a field of characteristic $2$.  We exhibit an explicit polynomial
endomorphism $F\colon \A_k^3\to\A_k^3$ whose Jacobian determinant is identically
$1$, whose induced extension of rational function fields has degree $3$, and
which is nevertheless noninjective.  Since $2\nmid 3$, this gives a
counterexample to the usual Adjamagbo, or separable, formulation of the
Jacobian conjecture in characteristic $2$.  Stabilization yields analogous
counterexamples in every dimension $n\geq 3$.
\end{abstract}

\maketitle

\section{Introduction}

Let $k$ be a field of characteristic $p>0$, let
$X=(x_1,\dots,x_n)$, and let
$F=(F_1,\dots,F_n)\in k[X]^n$.  Write
\[
  k(F)=k(F_1,\dots,F_n)\subseteq k(X)=k(x_1,\dots,x_n).
\]
A commonly used positive-characteristic refinement of the Jacobian
conjecture, originating with Adjamagbo, is the following; see
\cite{Adjamagbo1995,KimuraOkuda2005,MaubachRauf2017}.

\begin{conjecture}[Separable Jacobian conjecture]\label{conj:sjc}
Assume that
\[
  \det \Jac(F)\in k^\times
  \qquad\text{and}\qquad
  p\nmid [k(X):k(F)].
\]
Then $F$ is a polynomial automorphism of $\A_k^n$.
\end{conjecture}

The original Jacobian conjecture was shown to be false in dimension three by Anthropic \cite{Anthro}. The purpose of this note is to record the following explicit counterexample to the characteristic 2 case, the example given by Anthropic does not specialize, but a modification is sufficient to provide a counterexample for this case. As far as the author is aware, no other counterexample to this conjecture is known. 

\begin{theorem}\label{thm:main}
Let $k$ be any field of characteristic $2$, and define
$F=(P,Q,R)\colon\A_k^3\to\A_k^3$ by
\begin{equation}\label{eq:map}
  F(x,y,z)=
  \bigl(x+x^2y,\;y+xz+x^2yz,\;z+x^2z^2\bigr).
\end{equation}
Then
\[
  \det \Jac(F)=1,
  \qquad
  [k(x,y,z):k(P,Q,R)]=3,
\]
and $F$ is not injective.  Consequently, \cref{conj:sjc} is false in
characteristic $2$ and dimension $3$.
\end{theorem}
\section{AI Usage}
This text is the result of a discussion with ChatGPT 5.6 Sol. Conversation transcripts are available by request. The proof itself has been verified as correct by the named author.

\section{Invertibility}

\begin{proposition}\label{prop:jac-collision}
The map in \eqref{eq:map} satisfies $\det\Jac(F)=1$ and is not injective over the prime field $\Ftwo$.
\end{proposition}

\begin{proof}
In characteristic $2$, derivatives of squares vanish.  Hence
\[
\Jac(F)=
\begin{pmatrix}
1 & x^2 & 0\\
z & 1+x^2z & x+x^2y\\
0 & 0 & 1
\end{pmatrix},
\]
and therefore
\[
  \det\Jac(F)=(1+x^2z)-x^2z=1.
\]
Direct substitution gives
\[
  F(0,1,0)=F(1,1,0)=F(1,1,1)=(0,1,0).
\]
The three source points are distinct, so $F$ is not injective and hence is not
an automorphism.
\end{proof}

\section{The cubic function-field description}

Set
\[
  K=k(P,Q,R),
  \qquad
  L=k(x,y,z).
\]
Make the triangular change of target coordinates
\begin{equation}\label{eq:UVW}
  U=P,
  \qquad
  V=Q+PR,
  \qquad
  W=R.
\end{equation}
Since $Q=V+UW$ in characteristic $2$, one has
$K=k(U,V,W)$.

Consider the cubic
\begin{equation}\label{eq:D}
  D(T)=UT^3+T^2+VT+W\in K[T].
\end{equation}
The key point is that this cubic has a root that generates all of $L$ over
$K$.

\subsection{A linear--quadratic factorization}

Define
\begin{equation}\label{eq:bcd}
  b=1+x^2z,
  \qquad
  c=1+xy,
  \qquad
  d=Q.
\end{equation}
Then
\begin{equation}\label{eq:basic-identities}
  P=xc,
  \qquad
  R=zb,
  \qquad
  xd+bc=1.
\end{equation}
Moreover,
\begin{equation}\label{eq:Videntity}
  xz+bd=d+xczb=Q+PR=V.
\end{equation}
Indeed, the first equality in \eqref{eq:Videntity} follows from
$1+xd=bc$ by multiplying by $xz$.

It follows that
\begin{equation}\label{eq:factorization}
\boxed{
D(T)=(xT+b)(cT^2+dT+z).
}
\end{equation}
For later use, the two factors have resultant
\begin{align}
\Res_T(xT+b,cT^2+dT+z)
  &=x^2z+xbd+b^2c \notag\\
  &=x^2z+b(xd+bc) \notag\\
  &=x^2z+b=1.
\label{eq:resultant}
\end{align}

\subsection{A primitive element and rational reconstruction}

Put
\begin{equation}\label{eq:t}
  t=\frac{b}{x}=\frac{1+x^2z}{x}\in L.
\end{equation}
Because the characteristic is $2$, the linear factor in
\eqref{eq:factorization} vanishes at $T=t$, and hence $D(t)=0$.

\begin{lemma}\label{lem:reconstruct}
One has $L=K(t)$.  More explicitly, if
\[
  a=D'(t)=Ut^2+V,
\]
then
\begin{equation}\label{eq:reconstruction}
  x=a^{-1},
  \qquad
  z=a(t+a),
  \qquad
  y=Q+Pz.
\end{equation}
\end{lemma}

\begin{proof}
Write
\[
  \ell(T)=xT+b,
  \qquad
  h(T)=cT^2+dT+z,
\]
so that $D=\ell h$.  Since $\ell(t)=0$,
\[
  D'(t)=\ell'(t)h(t)=x h(t).
\]
For a linear polynomial and a quadratic polynomial,
\[
  \Res_T(\ell,h)=x^2h(t).
\]
Together with \eqref{eq:resultant}, this gives
\[
  D'(t)=x h(t)=\frac{1}{x}.
\]
Thus $a=1/x$ and the first formula in \eqref{eq:reconstruction} follows.
Next,
\[
  t=\frac{1+x^2z}{x}=\frac1x+xz=a+xz,
\]
so $z=a(t+a)$.  Finally,
\[
  Q=y+xz+x^2yz=y+(x+x^2y)z=y+Pz,
\]
whence $y=Q+Pz$.  Therefore $x,y,z\in K(t)$, and the reverse inclusion is
obvious.
\end{proof}

Since $t$ is algebraic over $K$ by \eqref{eq:D}, \cref{lem:reconstruct}
shows that $L/K$ is finite of degree at most $3$.  Consequently
\[
  \trdeg_k K=\trdeg_k L=3.
\]
As $K=k(U,V,W)$ is generated by three elements, $U,V,W$ are algebraically
independent over $k$.  In particular,
\[
  K\cong k(U,V,W)
\]
is a purely transcendental field.

\subsection{Irreducibility of the cubic}

\begin{lemma}\label{lem:irreducible}
The polynomial $D(T)=UT^3+T^2+VT+W$ is irreducible in $K[T]$.
\end{lemma}

\begin{proof}
First view $D$ as a polynomial in $W$ over the domain $k[U,V,T]$:
\[
  D=W+(UT^3+T^2+VT).
\]
It is monic of degree one in $W$.  In any factorization in
$k[U,V,T][W]$, one factor must have $W$-degree zero; comparison of the
coefficient of $W$ then forces that factor to be a unit.  Thus $D$ is
irreducible in $k[U,V,W,T]$.

As a polynomial in $T$ over the UFD $k[U,V,W]$, the coefficients of $D$ are
primitive because the coefficient of $T^2$ is $1$.  Gauss's lemma therefore
implies that $D$ is irreducible over the fraction field
$k(U,V,W)=K$.
\end{proof}

\begin{proposition}\label{prop:degree}
The extension $L/K$ has degree $3$ and is separable.
\end{proposition}

\begin{proof}
By \cref{lem:reconstruct}, $L=K(t)$, and $t$ is a root of $D$.  By
\cref{lem:irreducible}, $D$ is the minimal polynomial of $t$ over $K$ up to
multiplication by a nonzero scalar.  Hence
\[
  [L:K]=3.
\]
Furthermore, \cref{lem:reconstruct} gives
\[
  D'(t)=\frac1x\neq 0,
\]
so the extension is separable.
\end{proof}

\section{Consequences}

\begin{proof}[Proof of \cref{thm:main}]
The Jacobian and noninjectivity assertions are
\cref{prop:jac-collision}, while the function-field assertion is
\cref{prop:degree}.  Since $2\nmid 3$, all hypotheses of
\cref{conj:sjc} hold, but the conclusion fails.
\end{proof}

\begin{corollary}\label{cor:stabilization}
For every $n\geq 3$, the map
\[
  F_n=(P,Q,R,x_4,\dots,x_n)\colon\A_k^n\longrightarrow\A_k^n
\]
is a counterexample to the separable Jacobian conjecture in characteristic
$2$.
\end{corollary}

\begin{proof}
The Jacobian matrix is block diagonal with blocks $\Jac(F)$ and
$I_{n-3}$, so its determinant is $1$.  The collision in
\cref{prop:jac-collision} persists after fixing the remaining coordinates.
Finally, adjoining algebraically independent variables preserves finite field
extension degree, and therefore
\[
[k(x,y,z,x_4,\dots,x_n):k(P,Q,R,x_4,\dots,x_n)]=3.
\]
\end{proof}

\begin{remark}
The construction does not address the two-dimensional version of the
conjecture.
\end{remark}

\section{Origin of the construction}

The map can be recovered systematically from a normalized factorization of a
cubic.  Begin with
\[
  \ell(T)=xT+\beta,
  \qquad
  h(T)=\gamma T^2+\delta T+\varepsilon.
\]
Impose the normalization that the coefficient of $T^2$ in $\ell h$ is $1$:
\begin{equation}\label{eq:normalization}
  x\delta+\beta\gamma=1.
\end{equation}
The resultant is
\[
  \Res_T(\ell,h)=x^2\varepsilon+x\beta\delta+\beta^2\gamma.
\]
In characteristic $2$, equation \eqref{eq:normalization} reduces this to
\[
  \Res_T(\ell,h)=x^2\varepsilon+
  \beta(x\delta+\beta\gamma)=x^2\varepsilon+\beta.
\]
Thus the additional normalization $\Res_T(\ell,h)=1$ is solved by
\[
  \beta=1+x^2\varepsilon.
\]
Taking
\[
  \varepsilon=z,
  \qquad
  \gamma=1+xy,
  \qquad
  \delta=y+xz+x^2yz
\]
satisfies \eqref{eq:normalization}.  The coefficients of $\ell h$ are then
\[
  P,\quad 1,\quad Q+PR,\quad R,
\]
which yields exactly the cubic \eqref{eq:D} and the map \eqref{eq:map}.

\section{Exact symbolic verification}

An accompanying Python script performs exact computations over $\Ftwo$ using
SymPy.  It checks:
\begin{itemize}
  \item $\det\Jac(F)=1$;
  \item the explicit three-point collision;
  \item the factorization \eqref{eq:factorization};
  \item the resultant identity \eqref{eq:resultant}; and
  \item the denominator-cleared reconstruction identities.
\end{itemize}
The irreducibility argument in \cref{lem:irreducible} is independent of the
computer calculation.

\end{document}